\documentclass[11pt,reqno]{amsart}
\usepackage{amssymb}
\usepackage{amsfonts}
\usepackage{amsthm}
\usepackage{amsmath}
\usepackage[pagebackref,hypertex]{hyperref} 
\usepackage{backref}

\newcommand{\R}{{\mathbb{R}}}

\newcounter{smalllist}

\numberwithin{equation}{section}
\allowdisplaybreaks

\newtheorem{thm}{Theorem}
\newtheorem{rem}[thm]{Remark}
\newtheorem{lemma}[thm]{Lemma}

\newtheorem{estimate}[thm]{Estimate}

\newcommand{\bbr}{\mathbb R}
\newcommand{\bbz}{\mathbb Z}

\newcommand{\calb}{\mathcal B}

\newcommand{\calf}{\mathcal F}
\newcommand{\calg}{\mathcal G}
\newcommand{\cali}{\mathcal I}

\newcommand{\calp}{\mathcal P}

\newcommand{\calr}{\mathcal R}
\newcommand{\calt}{\mathcal T}

\newcommand{\calu}{\mathcal U}

\newcommand{\one}{\mathbf{1}}

\newcommand{\topp}{\mathbf {top} }

\newcommand{\beq}{\begin{equation}}
\newcommand{\eeq}{\end{equation}}
\newcommand{\beqa}{\begin{eqnarray*}}
\newcommand{\eeqa}{\end{eqnarray*}}
\newcommand{\beqan}{\begin{eqnarray}}
\newcommand{\eeqan}{\end{eqnarray}}

\begin{document}

\title[Hilbert transform along one-variable vector fields   ] {$L^p$ estimates for the Hilbert transforms along a one-variable vector field}
\author{Michael Bateman \and Christoph Thiele}
\address{M. Bateman, Department of Mathematics, UCLA, Los Angeles, CA 90095-1555;
\email{bateman@math.ucla.edu}
\and
C. Thiele
, Department of Mathematics, UCLA, Los Angeles, CA 90095-1555;
\email{thiele@math.ucla.edu}\ 
}

\begin{abstract}
Stein conjectured that the Hilbert transform in the direction of a vector field is bounded on, say, $L^2$ whenever $v$ is Lipschitz.  We establish a wide range of $L^p$ estimates for this operator when $v$ is a measurable, non-vanishing, one-variable 
vector field in $\bbr ^2$.  Aside from an $L^2$ estimate following from a simple trick with Carleson's theorem, these estimates were unknown previously.  This paper is closely related to a recent paper of the first author (\cite{B2}).

\end{abstract}

\maketitle



\section{Introduction}

Given a non-vanishing measurable vector field $v \colon \bbr ^2 \rightarrow \bbr^2 $, define 
for $f \colon \bbr ^2 \rightarrow \bbr^2 $
\beqan\label{ht}
H_{v} f(x,y) = p.v. \int {{f((x,y)-tv(x,y)) } \over t}\, dt\ \ .
\eeqan
In this paper we prove: 
\begin{thm} \label{main}
Suppose $v$ is a non-vanishing measurable vector field such that
for all $x,y\in \bbr$ 
\beqa
v(x,y) = v(x,0)\ \ ,  
\eeqa
and suppose $p \in ({3\over 2} , \infty )$.  Then 
\beqa
||H_{v} f || _p \lesssim ||f||_p\ \ .
\eeqa
\end{thm}

The estimate is understood as an a priori estimate for all $f$ in 
an appropriate dense subclass of $L^p(\bbr^2)$, say the Schwartz class, 
on which the Hilbert transform $H_v$ is initially defined. One can then 
use the estimate to extend $H_v$ to all of $L^p(\bbr^2)$.

If the vector field is constant, then this follows from classical estimates
for the one dimensional Hilbert transform by evaluating the $L^p$ norm as an iterated 
integral, with inner integration in the direction of the vector field.
Theorem \ref{main} follows from the special case for vector fields mapping to vectors of unit length, because 
the Hilbert transforms along $v$ and $\frac{v}{|v|}$ are equal
by a simple change of variables in (\ref{ht}).
To prove the theorem for unit length vector fields, it suffices to do so for
vector fields with non-vanishing first component, because we can apply the result for constant 
vector fields to the restriction of $H_v$ to the set where $v$ takes the value $(0,1)$ 
and the set where it takes the value $(0,-1)$. Dividing $v$ by its first
component we may then assume it is of the form $(1,u(x))$; note that multiplying $v$ by a 
negative number merely changes the sign of (\ref{ht}). We call $u$ the slope of the vector field. 
The Hilbert transform (\ref{ht}) then takes the form
\beqan \label{hv}
H_{v} f(x,y) = p.v. \int {{f(x-t,y-t u(x)) } \over t} \, dt\ \ .
\eeqan


\subsection{Remarks and related work}

The case $p=2$ of Theorem \ref{main} is equivalent to the  Carleson-Hunt 
theorem in $L^2$. This observation is attributed (without reference) 
to Coifman in \cite{LL2} and to Coifman and El Kohen in \cite{CSWW}.
We briefly explain how to deduce Theorem \ref{main} for $p=2$ from the 
Carleson-Hunt theorem.
Denote by ${\calf}_2$ the Fourier transform in the second variable.  Then
we formally have for (\ref{hv}), ignoring principal value notation,
\beqa
\int e^{2\pi i \eta y} \int \calf_2 f(x-t,\eta) {e^{-2\pi i u(x)\eta t} \over t} dt\, d\eta \ \ . 
\eeqa
As the inner integral is independent of $y$, it suffices by Plancherel to prove
\beqa
\| \int \calf_2f(x-t,\eta) {e^{-2\pi i u(x)\eta t} \over t} dt\|_{L^2(x,\eta)} 
\lesssim \|\calf_2 f\|_2\ \ . 
\eeqa
Applying for each fixed $\eta$ the Carleson-Hunt theorem in the form 
\beqa
\| \int g(x-t) {e^{-2\pi i N(x) t} \over t} dt\|_2 \lesssim \|g \|_2\ \  ,
\eeqa
for $g\in L^2(\bbr)$ and measurable function $N$ proves
the desired estimate.

For any regular linear transformation of the plane we have the identity
$$(H_{T\circ v \circ T^{-1}} f)\circ T= H(f\circ T)\ \ .$$ 
The class of vector fields depending on the first variable is invariant
under linear transformations which preserve the vertical direction.
This symmetry group is generated by
the isotropic dilations
$$(x,y)\to ( \lambda x,\lambda y)\ \ ,$$ 
non-isotropic dilations
$$(x,y)\to (x,\lambda y)\ \ ,$$ 
and the shearing transformations
$$(x,y)\to (x,y+\lambda x)$$
for $\lambda\neq 0$.
By a simple limiting argument, it suffices to prove Theorem \ref{main}
under the assumption that $\|u\|_\infty$ is finite. By the above non-isotropic scaling
the operator norm is independent of $\|u\|_\infty$, and we may therefore
assume without loss of generality that 
\beqan \label{uassume}
\|u\|_\infty\le 10^{-2}\ \ .
\eeqan

Following general principles of wave packet analysis, it is 
natural to decompose $H_v$ into wave packet components, where
the wave packets are obtained from a generating function $\phi$
via application of elements of the symmetry group
of the operator.
These wave packets can be visualized by acting with the same
group element on the unit square in the plane. The shapes obtained
under the above linear symmetry group of $H_v$ are parallelograms with a 
pair of vertical edges. All parallelograms in this paper will be of this
special type. Under the assumption (\ref{uassume}) it suffices
to consider parallelograms whose non-vertical edges are close to horizontal. 
Such parallelograms 
are well approximated by rectangles, which are used in \cite{B2} and previous
work by Lacey and Li \cite{LL2}.

The companion paper \cite{B2} proves the following theorem:
\begin{thm}\label{companion}
Assume $\|u\|_\infty\le 1$ and $1<p<\infty$. Assume $\widehat{f}(\xi,\eta)$
vanishes outside an annulus $A<|(\xi,\eta)|\le 2A$ for some $A>0$. Then
\beqa
||H_v f || _p \lesssim ||f||_p\ \ .
\eeqa
\end{thm}
Actually, the theorem is stated there for functions such that $\widehat{f}$ vanishes outside a trapezoidal region inside an annulus, but this is inessential, as can be seen from the commentary below.  
This theorem is weaker than Theorem \ref{main} in the region
$p>3/2$ but stronger in the region $1<p\le 3/2$. 
The width of the annulus can be altered
by finite superposition of different annuli, at the expense of
an implicit constant depending on the conformal width of the annulus.
The case $p> 2$ and a weak type endpoint at $p=2$ of Theorem \ref{companion} 
are due to Lacey and Li \cite{LL1} , and hold
for arbitrary measurable vector fields.

We reformulate Theorem \ref{companion} in a form invariant 
under the above linear transformation group.
Note that the adjoint linear transformations of this group
leave the horizontal direction invariant.
\begin{thm}\label{companion2}
Assume $1<p<\infty$. Assume $\widehat{f} (\xi, \eta)$
is supported in a horizontal pair of strips  $A< |\eta| < 2A$ for
some $A>0$. Then
\beqa
||H_v f || _p \lesssim ||f||_p\ \ .
\eeqa
\end{thm}
To deduce Theorem \ref{companion2} from Theorem \ref{companion}
we use the non-isotropic dilation $(x,y)\to (\lambda  x,y)$ to stretch the
annulus in $\xi$ direction until in the limit it degenerates to a pair of strips
$A<|\eta|<2A$. The restriction $\|u\|_\infty\le \lambda^{-1}$ becomes void in the limit
$\lambda\to 0$. This proves Theorem \ref{companion2}. 
For the converse direction  we use a bounded number of dilated strips to cover the annulus 
except for two thin annular sectors around the $\xi$-axis. It remains to prove
bounds on functions supported in these sectors. For fixed constant vector $v$, the operator 
$H_v$ is given by a Fourier multiplier which is constant on two half planes separated
by a line through the origin perpendicular to $v$. If $\|u\|_\infty\le 1$, then this line does 
not intersect the thin annular sectors, and we have with the constant vector field $(1,0)$:
\beq 
\label{honezero}
H_vf(x,y)=H_{(1,0)}f(x,y)\ \ .
\eeq 
But $H_{(1,0)}$ is trivially bounded and this completes the deduction of Theorem \ref{companion}
from Theorem \ref{companion2}.

Sharpness of the exponent ${3\over 2}$ in Theorem \ref{main} is not known. 
In Remark \ref{othermaximal} we mention a potential covering lemma that, when combined with the 
methods in this paper, would push the exponent down to ${4\over 3}$.  The truth of this covering lemma is 
unknown, however.  
If $f$ is an elementary tensor, $$f(x,y)=g(x) h(y)\ ,$$ then a similar calculation as above turns
$H_vf$ into
\beqa
\int \widehat{h}(\eta )e^{2\pi i \eta y} \int g(x-t) {e^{-2\pi i u(x)\eta t} \over t} dt\, d\eta \ \ . 
\eeqa
This expression can be read as 
a family of Fourier multipliers acting on $h$. Assuming the norm of $h$ is normalized
to $\|h\|_p=1$, we can estimate the last display by
$$\|\|\int g(x-t) {e^{-2\pi i u(x) \eta t} \over t} \, dt\|_{M^{p}(\eta)}\|_{L^p(x)}\ \ ,$$
where $M_p(\eta)$ denotes the operator norm of the Fourier multiplier acting on $L^p$. 
By scaling invariance of the multiplier norm, the factor $u(x)$ in the phase can
be ignored. As shown in (\cite{CRS}), multiplier norms are controlled by variation norms.
Hence we may estimate the last display by
$$\|\|\int g(x-t) {e^{-2\pi i \eta t} \over t} \, dt\|_{V^{r}(\eta)}\|_{L^p(x)}\ \ ,$$
provided $|\frac 12-\frac 1p|\le \frac 1r$. The bounds on the variation norm Carleson operator
in \cite{OSTTW} imply that for $p>{4\over 3}$ and $r>p'$ the last display is bounded by a constant 
times $\|g\|_p$. Hence the exponent in Theorem \ref{main} can be improved to ${4 \over 3}$ under the
additional assumption that the function $f$ is an elementary tensor.
The authors learned this argument from Ciprian Demeter. 
Related multiplier theorems in \cite{DLTT}, \cite{D} also show a phase transition at this exponent.

The Hilbert transform along a one variable vector field has previously 
been studied by Carbery, Seeger, Wainger, and Wright in \cite{CSWW}.
There boundedness in $L^p$ for $1<p$ is proved under additional
conditions on the vector field.

In a different direction, Stein has conjectured that a truncation of $H_v$ is bounded 
on $L^2$ under the assumption that  the two-variable vector field $v$ is Lipschitz with
sufficiently small Lipschitz constant depending on the truncation.  
Stein's conjecture is related to a well-known conjecture of Zygmund on the differentiation 
of Lipschitz vector fields.  Define 
\beqa
M_{v} 
f (x,y) = \sup _{0<L<1} {1\over {2L}} \int _{-L} ^L f((x,y) - v(x,y)t) dt .
\eeqa
Zygmund conjectured that $M_v$  is (say) weak-type (2,2) if $\|v\|_\infty$ is bounded and 
the Lipschitz norm $\|\nabla v\|_\infty$ is small enough.  
Proving a weak-type estimate on this operator would yield corresponding differentiation results analogous to the Lebesgue differentiation theorem, except the averaging takes place over line segments instead of balls.  Estimates on $M_v $ are unknown on any $L^p$ space, except for the trivial $p=\infty$ case, unless more stringent requirements are placed on $v$; for example, Bourgain \cite{BO} proved $M_{v} $ is bounded on $L^p$, $p>1$ when ${v} $ is real-analytic and the operator is restricted to a bounded domain. The corresponding result for the Hilbert transform is announced in \cite{SS}, although the $p=2$ case follows from work of Lacey and Li \cite{LL2}.
Previously the Hilbert transform case, in such a range of exponents,  was only known (\cite{CNSW})
under the additional assumption that no integral curve of the vector field forms a straight line.

There is some history of using singular integral and time-frequency methods to control positive maximal operators.  See Lacey's bilinear maximal theorem (\cite{L}) or 
the extension \cite{DLTT} of Bourgain's return times theorem by Demeter, Lacey, Tao, and the second author.

This paper is structured as follows: Section \ref{vvreduction} contains the main approach: a separation of frequency space into horizontal
dyadic strips and application of Littlewood-Paley theory in the second variable to reduce to some 
vector-valued inequality; this step uses 
the one-variable property of the vector field to ensure that the strips are invariant under $H_v$. This fact
was brought
to our attention by Ciprian Demeter. The vector-valued inequality is proved by restricted weak-type interpolation,
a tool that allows to localize the operator to some benign sets $G$ and $H$ and prove
strong $L^2$ bounds on these sets.

Section \ref{constructsets} gives the crucial construction of the sets $G$ and $H$ relying
on two covering lemmas. One is essentially an argument by Cordoba and R. Fefferman \cite {CF},
while the other is essentially an argument by Lacey and Li \cite{LL3}.

Section \ref{outline} outlines the proof of the $L^2$ bounds on the sets $G$ and $H$ using
time-frequency analysis as in \cite{B2}. The operator that we estimate at this point is a refinement of the operator in \cite{B2}. We refer to the decomposition of this operator in \cite{B2}
without recalling details. The terms in this decomposition satisfy Estimates \ref{orthogonalityct} 
through \ref{trivialsize}, which are also taken from \cite{B2}. To complete the 
proof of Theorem \ref{main}, we need the additional
Estimates \ref{secondmaximal} and \ref{sizerestriction}, which depend on the sets $G$ and $H$.
These additional estimates are proved in Section \ref{fillingdetails}, again with much 
reference to \cite{B2}.

Throughout the paper, we write $x\lesssim y$ to mean there is a universal constant $C$ such that $x\leq Cy$.  We write $x\sim y$ to mean $x\lesssim y$ and $y\lesssim x$.  We write $\one _E$ to denote the characteristic function of a set $E$.


\section{Reduction to estimates for a single frequency band} \label{vvreduction}

We fix the vector field $v$ with the normalization (\ref{hv}) and assume bounded slope as in (\ref{uassume}).
Let $P_c$ be the Fourier restriction operator to a double cone as follows:
$$\widehat{P_c f}(\xi,\eta)=
\one_{10|\xi|\le |\eta|}\widehat{f}(\xi,\eta)\ .$$
It suffices to estimate $H_vP_c$ in place of $H_v$ because, similarly to
(\ref{honezero}),
$$H_v(\one -P_c) f(\xi,\eta)=  H_{(1,0)}(\one -P_c)f(\xi,\eta)\ \ ,$$
due to the restriction on the slope of $v$.
Define the horizontal pair of bands
\beqa
B_k := \{ (\xi , \eta ) \in \bbr ^2 \colon |\eta| \in [2^{k} , 2^{k+{1\over 100}} ) \}\ \ ,
\eeqa
and define the corresponding Fourier restriction operator
$\widehat{ P_k f} = \one _{B_k} \hat{f}$.
Since the Hilbert transform in a constant direction is
given by a Fourier multiplier, and the vector field $v$ is
constant on vertical lines, we can formally write for a family
of multipliers parameterized by $x$:

$${H_v f}(x,y)= \int\int m_x (\xi,\eta)\widehat{f}(\xi,\eta)
e^{2\pi i (x\xi+y\eta)}\, d\xi d\eta \ \ .$$
Then it is clear that 
\beqa {H_v(P_k f)}(x,y) &=& \int   1_{[2^k,2^{k+{1\over 100}})}(\eta) e^{2\pi i y\eta}
[\int m_x (\xi,\eta)\widehat{f}(\xi,\eta)
e^{2\pi i x\xi} \, d\xi] d\eta \\
&=& P_k (H_v f)(x,y)
\ \ .
\eeqa
Define  
$$H_k:= P_k H_v P_c = P_k H_v P_c P_k\ \ .$$
Littlewood-Paley theory implies
\beqa
|| H_v P_c f || _p \lesssim || \left(  \sum _{k\in \bbz/100} |H_k f |^2 \right) ^{1\over 2} || _p \ \ ,
\eeqa
where here the summation is over integer multiples of ${1\over 100}$.
Using Littlewood-Paley theory once more, it suffices to prove
\beqa 
|| \left(  \sum _{k\in \bbz/100} |H_k (P_k f) |^2 \right) ^{1\over 2} || _p 
\lesssim 
|| \left(  \sum _{k\in \bbz/100} |P_kf |^2 \right) ^{1\over 2} || _p \ \ ,
\eeqa
which follows from the more general estimate
\beqa 
|| \left(  \sum _{k\in \bbz/100} |H_k f_k |^2 \right) ^{1\over 2} || _p 
\lesssim 
|| \left(  \sum _{k\in \bbz/100} |f_k |^2 \right) ^{1\over 2} || _p \ \ 
\eeqa
for any sequence of functions $f_k\in L^2$. 
By a limiting argument, it suffices to prove for all $k_0>0$ 
\beqan \label{vvinequality} 
|| \left(  \sum _{ 
|k|\le k_0} |H_k f_k |^2 \right) ^{1\over 2} || _p 
\lesssim 
|| \left(  \sum _{|k|\le k_0} |f_k |^2 \right) ^{1\over 2} || _p \ \ 
\eeqan
with implicit constant independent of $k_0$, where it is understood
that $k$ runs through elements of $\bbz/100$. Compare this inequality with
a vector valued Carleson inequality as in \cite{GMS}.

Theorem (\ref{companion2}) implies that $H_k$ is bounded in 
$L^p$ for $1<p<\infty$ for each $k$. In particular, (\ref{vvinequality}) 
is true for $p=2$ by interchanging the order of square summation and $L^2$ norm.

Note that $H_k$ is defined a priori on all of $L^p$ (by Theorem \ref{companion2}) and we may drop the assumption
that $f$ is in the Schwartz class.
By Marcinkiewicz interpolation for $l^2$ vector valued functions it suffices to prove 
for $G,H \subseteq \bbr ^2$ and $\sum _k | f_k |^2  \leq \one _H$:  
\beqan \label{restrictedweakvv}
| \langle \left(  \sum _{|k|\le k_0} | H_k f_k |^2 \right) ^{1\over 2}, \one _G \rangle | 
\lesssim |H| ^{1\over p} |G| ^{1- {1\over p} }\ \ .
\eeqan
By Lebesgue's monotone convergence theorem it suffices to prove
this under the assumption that $G$ is supported on a large square $[-N',N']^2$
as long as the implicit constant does not depend on $N'$. By another
limiting argument using crude estimates in case the sets $G$ and $H$ have large
distance it suffices to prove this
under the assumption that $H$ is supported in a much larger square $[-N,N]$,
again with bounds independent of $N$. Generalizing, we will only assume
both $G$ and $H$ are supported on the larger square.

Since we already have (\ref{restrictedweakvv}) for $p=2$, 
we immediately obtain this estimate for $p>2$ provided $|H|\lesssim |G|$ and for $p<2$ provided $|G|\lesssim |H|$.
By a standard induction on the ratio of $|H|$ and $|G|$, it then suffices to prove the following lemma.

\begin{lemma}\label{majorsubsetlemma} 
Let $G',H' \subset [-N,N]^2$ be measurable 
and let $\frac 32<p<\infty$.

If $p>2$ and $10|G'|< |H'|$ then there exists a subset $H\subset H'$ depending only on
$p$, $G'$, and $H'$ with $|H|\ge |H'|/2$ 
such that (\ref{restrictedweakvv}) holds with $G=G'$ and any
sequence of functions $f_k$ with $\sum_{|k|\le k_0} |f_k|^2\le \one_H$. 

If $p<2$ and $10|H'|< |G'|$ then there exists a subset $G\subset G'$ depending only on
$p$, $G'$, and $H'$ with $|G|\ge |G'|/2$ 
such that (\ref{restrictedweakvv}) holds with $H=H'$ and any
sequence of functions $f_k$ with $\sum_{|k|\le k_0} |f_k|^2\le \one_H$. 
\end{lemma}

For example in case $p>2$ and $10|G'|<|H'|$ we split $H'$ into $H$ and $H'\setminus H$
and apply the triangle inequality. On $H'\setminus H$ we apply the induction hypothesis, 
which yields an estimate better than the desired one by a factor $2^{-1/p}$ because of the 
size estimate for $H'\setminus H$. On $H$ we use the conclusion of the Lemma, which by choosing 
the induction statement properly we may assume to provide a bound no more than $1-2^{-1/p}$ times 
the desired bound.

By Cauchy Schwarz, (\ref{restrictedweakvv}) follows from
\beqa
\int \sum _{|k|\le k_0} 
| H_k f_k |^2 \one _G \lesssim |H|^{2\over p} |G| ^{1-{2\over p}}\ \ . 
\eeqa
This in turn follows from
\beqan\label{strongghfe}
\int \sum _{|k|\le k_0}  | H_k f_k |^2 \one _G 
\lesssim \left( {{|G|}\over {|H|}} \right) ^{1-{2\over p}} \int \sum _k  |f_k |^2 
\eeqan
by the assumption on the sequence $f_k$.
Now define the operator $H_{k,G,H}$ by
$$H_{k,G,H} f= \one_G  H_k (\one_H f)\ \ .$$
Then (\ref{strongghfe}) follows from the estimate
$$\|H_{k,G,H} f\|_2\lesssim \left( {{|G|}\over {|H|}} \right) ^{{1\over 2}-{1\over p}} \|f\|_2\ .$$
for any $f\in L^2$, and $|k|\le k_0$, 
 assuming the implicit constant does not depend on $k$ or $k_0$.
We will prove this $L^2$ estimate again by Marcinkiewicz interpolation
between weak type estimates. More precisely we will prove

\begin{thm}\label{firstrefinement}
Let $p$ be as in Theorem \ref{main} and let $G', H' \subseteq \bbr ^2$ be as in Lemma \ref{majorsubsetlemma}.
Then there are sets $G,H$ as in Lemma \ref{majorsubsetlemma} such that 
for any measurable sets $E, F \subset \bbr^2$ and each $|k|\le k_0$ we have
\beqan\label{firstrefin}
|\langle H_{k,G,H} \one _F , \one _E \rangle| \lesssim
\left( {{ |G| } \over {|H|} }\right) ^{{1\over 2}-{1 \over p}}  |F|^{{1\over 2} } |E|^{{1\over 2}}\ \ .\eeqan
\end{thm}
Note again that \cite{B2} proves 
\beqan\label{wtfe}
|\langle H_{k,G,H} \one _F , \one _E \rangle| \lesssim
|F|^{{1\over q} } |E|^{1- {1\over q}}
\eeqan
for all $1<q<\infty$.  The refinement we need here is the localization to $G$ and $H$, with corresponding improvement in the estimate.  We remark that the parameter $k$ is irrelevant in proving \eqref{firstrefin}, but it is crucial that the sets $H$ and $G$ be constructed independent of $k$.
By interpolating Theorem \ref{firstrefinement} with (\ref{wtfe}) for $q$
near $1$ and $\infty$ we obtain strong type estimates 
$$
|\langle H_{k,G,H} f, e \rangle| \lesssim
\left( {{ |G| } \over {|H|} }\right) ^{{1\over 2}-{1 \over r} }
\|f\|_q \|e\|_{q'}\ \ .$$
where $r$ is as close to $p$ as we wish and $q$ is in a small punctured 
neighborhood of $2$ whose size depends on $r$. Another interpolation allows $q$ to be $2$ as well,
and we obtain (\ref{strongghfe}) with power $r$ instead of $p$, which
is no harm since we seek an open range of exponents.  
We have thus reduced Theorem \ref{main} to Theorem \ref{firstrefinement}.


\section{Construction of the sets $G$ and $H$}\label{constructsets}

In this section we present the sets $G$ and $H$ of Lemma \ref{majorsubsetlemma}
 and prove the size estimates
$|G|\ge |G'|/2$ and $|H|\ge |H'|/2$.
Inequality (\ref{firstrefin}) will be proved in subsequent sections.

We work with two shifted dyadic grids on the real line,
\beqa
\cali_1&=&\{[2^k (n+\frac{(-1)^k}{3}), 2^k (n+1+\frac{(-1)^k}{3})): k,n\in \bbz\}\ \ , \\
\cali_2&=&\{[2^k (n-\frac{(-1)^k}{3}), 2^k (n+1-\frac{(-1)^k}{3})): k,n\in \bbz\}\ \ .
\eeqa
The exceptional sets will be the union of two sets:
\beqa 
H'\setminus H&=&H_1\cup H_2\ \ ,\\
G'\setminus G&=&G_1\cup G_2\ \ .
\eeqa
Fix $i\in \{1,2\}$. The sets $H_i$ and $G_i$ will be constructed using the grid $\cali_i$,
and we will prove $4|H_i|\le |H'|$ and $4|G_i|\le  |H'|$.

Given a parallelogram with two vertical edges, we define the height
$H(R)$ of the parallelogram to be the common length of the two vertical
edges. We define the shadow $I(R)$ to be the projection of $R$ onto
the $x$ axis. The central line segment of $R$ is the line 
segment which connects the midpoints of the two vertical edges. 
If a line segment can be written 
$$\{(x,y): x\in I(R):  y=ux+b\}\ \ ,$$
then we call $u$ the slope of the line segment.
For each parallelogram $R$ let $U(R)$ be the set of slopes
of lines which intersect both vertical edges. Note that
maximal and minimal slopes in $U(R)$ are attained by the diagonals
of the parallelogram. Hence $U(R)$ is an interval of length $2H(R)/|I(R)|$ centered
at the slope of the central line of $R$. 

For an interval $U$ and a positive number $C$ 
define $CU$ to be the interval with the same center but length
$C|U|$. If $R$ is a parallelogram, define 
$CR$ to be the parallelogram with the same central line 
segment as $R$ but height $CH(R)$ (this definition of $CR$ is used in
Section \ref{constructsets} only). Note that $CU(R)=U(CR)$. 
For an interval $I\subset I(R)$ define
\beqa
R_I = R \cap (I \times \bbr )\ \ .
\eeqa

Given $N$ and $k_0$ as in Lemma \ref{majorsubsetlemma}, 
we consider a finite set $\calr_i$ of parallelograms $R$ as follows:
the projection of both vertical edges of $R$ onto the  
$y$-axis are in $\cali_1\cup \cali_2$, and $I(R)\in \cali_i$. 
Further, the parallelogram is contained in the square $[-10^2N,10^2N]^2$,
the height is at least $2^{-k_0}$, and the slope is at most $10^{-1}$.
These assumptions imply also that $|I(R)|$ is also at least $2^{-k_0}$.

We will use the following simple geometric observation:
\begin{lemma}\label{7rlemma}
Let $R,R'$ be two parallelograms and assume
$I(R)= I(R')$,
$U(R)\cap U(R')\neq\emptyset$, $R\cap R'\neq \emptyset$,
and without loss of generality $H(R)\le H(R')$. Then we have
$R\subseteq 7R'$. Moreover, if  $7H(R)\le H(R')$, then
$7R\subseteq 7R'$.
\end{lemma}
\begin{proof}
Since $U(R)\cap U(R')\neq \emptyset$, there exist two parallel lines,
one intersecting both vertical edges of $R$ 
and the other intersecting both vertical edges of $R'$. Since $R\cap R'\neq\emptyset$,
the vertical displacement of these lines is less than $H(R)+H(R')$.  If $H(R)\le H(R')$,
then the vertical edges of $R$ have distance at most $2H(R')$ from
the respective vertical edges of $R'$ and are contained in the
vertical edges of $7R'$. This proves the first statement of the
lemma. The second statement follows similarly.
\end{proof}

Let $M_V$ denote the Hardy Littlewood maximal operator in vertical direction:
$$M_Vf(x,y)=\sup_{y\in J} \frac 1{|J|}\int_J |f(x,z)|\, dz\ \ ,$$
where the supremum is taken over all intervals $J$ containing $y$.
For a measurable function $u:\bbr \to \bbr$ (which will be the slope function
associated with the given vector field), define
$$E(R):=\{(x,y)\in R: u(x)\in U(R)\}\ \ .$$

\subsection{Construction of the set $H$}

With the sets $G',H'$ as in Lemma \ref{majorsubsetlemma}, we define
$$H_i=\bigcup \{R\in \calr_i: |E(R)\cap G'|\ge \delta |R| \}$$
with
$$\delta=C_\alpha\left({|G'|\over |H'|}\right)^{1-\alpha}$$
for some small $\alpha$ to be determined later through application
of Estimate \ref{sizerestriction}
and some constant $C_\alpha$ large enough so that the desired estimate 
$4|H_i|\le |H'|$ follows from
the following lemma, applied with $G=G'$, 
$q={1\over 1-\alpha}$.  Note that we are essentially eliminmating all rectangles $R$ with large density parameter, where density has the meaning from \cite{B2}.  This will be used in the proof of Estimate \ref{sizerestriction} later in the paper.  Essentially, trees with density $\geq \delta$ will have extremely small size, and will therefore be mostly negligible.

\begin{lemma}\label{shadow}
Let $\delta>0$ and $q>1$ and let $G\subset \bbr^2$ be a measurable set
and $u:\R\to \R$ be a measurable function.
Let $\calr$ be a finite collection of parallelograms with vertical
edges and dyadic shadow such that
$$|E(R)\cap G|\ge \delta |R| $$
for each $R\in \calr$. Then
$$\left|\bigcup_{R\in \calr}\right|\lesssim \delta^{-q}|G|\ \ .$$
\end{lemma}

\begin{proof} We will find a subset $\calg \subset  \calr$ such that 
\beqan \label{bad}
| \bigcup _{R \in \calr} R |
           & \lesssim  \sum_{R\in\calg } |R|\ \ ,\\
 \label{good}
\int ( \sum _{R\in\calg } \one _{E(R)} )^{q'} 
& \lesssim \sum_{R\in\calg } |R|\ \ .
\eeqan

Inequality (\ref{bad}) will complete the proof of Lemma 
\ref{shadow} provided
\begin{equation}\label{goodct}
\sum_{R\in\calg} |R| \lesssim \delta ^{-q }  |G|\ \  .
\end{equation}
But with the density assumption for the parallelograms in $\calr$ we have
\beqa
\sum_{R\in\calg} |R|  \leq \sum_{R\in\calg} {1\over {\delta}} |E(R)\cap G| 
&=& {1\over {\delta} } 
\left\| \sum_{R\in\calg} \one _{E(R)} \one_G\right\|_1 \\
&\lesssim & {1\over {\delta} } \left(\sum_{R\in\calg} |R| \right) ^{1/q'} |G|^{1/q}\ \ ,
\eeqa
where in the last line we have used  H\"older's inequality and (\ref{good}).
After division by the middle factor of the right hand side we obtain
(\ref{goodct}).

The following argument is essentially the one used in \cite{CF} to prove endpoint estimates for the strong maximal operator. We select parallelograms according to the following iterative procedure.  
Initialize
\beqa
STOCK &=& \calr \\
\calg &=& \emptyset \\
\calb &=& \emptyset .
\eeqa
While $STOCK \neq \emptyset$, choose an $R\in STOCK$ with  maximal $|I(R)|$.
If
\beqan \label{condition}
\sum _{R'\in \calg \colon E(R) \cap E({R'}) \neq \emptyset }
|7R\cap 7R'| \geq {10^{-2}} |R|\ ,
\eeqan
then update
\beqa
STOCK &:=& STOCK \setminus R \\
\calg &:=& \calg \\
\calb &:=& \calb \cup \{R\}\ .
\eeqa
Otherwise update 
\beqa
STOCK &:=& STOCK \setminus R \\
\calg &:=& \calg \cup \{R\}\\
\calb &:=& \calb \ .
\eeqa
It is clear that this procedure yields a partition $\calr = \calg \sqcup \calb$.  

To prove (\ref{bad}), let $R\in \calb$ and let $R'$ be in the set $\calg(R)$ 
of all elements in $\calg$ which are chosen prior to $R$ and
satisfy $E(R)\cap E(R')\neq \emptyset$. The last property implies $U(R)\cap U(R')\neq \emptyset$
and $R\cap R'\neq \emptyset$. Note also that $I(R)\subset I(R')$. By Lemma \ref{7rlemma} applied
to $R$ and $R'_{I(R)}$, we have for every
vertical line $L$ through the interval $I(R)$:
$$
|L\cap 7R\cap 7R' | \geq \min(H(R),H(R')\geq {{|7R\cap 7R' |}\over{7 |I(R)|}}\ \ .
$$
Comparing for $(x,y)\in R$ and corresponding vertical line $L$ the maximal function $M_V$ 
with an average over the segment $L\cap 7R$ we obtain:
\beqa
M_{V} ( \sum_{R'\in\calg(R)} \one _{7R'}) (x,y) &\geq & 
 7^{-1}H(R)^{-1} \sum_{R'\in\calg(R)} |L\cap 7R\cap 7R'|
\\ 
&\geq & 49^{-1}|R|^{-1} \sum_{R'\in\calg(R)} | 7R\cap 7R'|  
 \geq 
{10^{-4}}\ \ ,
\eeqa
where the last estimate followed from (\ref{condition}). Hence
\beqa
\left| \bigcup _{R\in\calb} R \right| \leq |\{ x\colon M_{V} ( \sum_{r\in\calg} \one _R) (x) \geq {10^{-4}}\}| \lesssim \sum _{R\in\calg } |R| 
\eeqa
by the weak $(1,1)$ inequality for $M_{V}$.  This proves \eqref{bad}, because
the corresponding estimate for the union of elements in $\calg$ is trivial.

To prove (\ref{good}), consider $R', R \in \calg$ with $E(R)\cap E(R') \neq \emptyset$. If  $R'$ was selected first, then $H(R) > 7H(R')$, for otherwise 
we can use Lemma \ref{7rlemma} as above to conclude for $(x,y)\in R$
\beqa
M_{V} (\one _{7R'}) (x,y) &\geq & 
 7^{-1}|H(R)|^{-1} \sum_{R'\in\calg(R)} |L\cap 7R\cap 7R'| \geq 
{49^{-1}}\ \ ,
\eeqa
and hence $R$ would have been put into $\calb$.
Hence we have by Lemma \ref{7rlemma} 
\beqan\label{fiveinfivect}
7R'_I \subset 7R _{I} 
\eeqan
for every $I\subset  I(R)$. Hence
$$\sum_{R' \in \calg (R)} |7R'_{I} \cap 7R_I | = 
\sum_{R' \in \calg (R)} |7R'_{I}|
$$
is proportional to $|I|$ for $I\subset I(R)$. Hence
we have for all such $I$
\beqan\label{fivelessfivect}
\sum_{R' \in \calg (R)} |7R'_{I} \cap 7R_I | \lesssim |R_I|\ ,
\eeqan
since for $I=I(R)$ this holds when condition 
\eqref{condition} fails.

Let's say an $n$-tuple $(R^1, R^2, \dots, R^n)$ of elements in $\calg$
is {\it admissible} if $R^j$ is selected after $R^{j+1}$ for each $j$ and
$E(R^{j})\cap E(R^{j+1})\neq \emptyset$.
Then we have
\beqa
\int \left( \sum_{R\in \calg} \one _{E(R)} \right) ^n 
&\lesssim &  \sum _{R^1, \dots, R_n} |E({R^1}) \cap E({R^2}) 
\cap \dots \cap E({R^n})| \\
&\lesssim & \sum_{(R^1, R^2, \dots, R^n) \text{ adm.} } 
	|E({R^1}) \cap E({R^2}) \cap \dots \cap E({R^n})| \\
&\lesssim & \sum_{(R^1, R^2, \dots, R^n) \text{ admi.} } 
	|7R^1 \cap 7R^2 \cap \dots 7R^n |.\\
&\lesssim & \sum_{(R^1, R^2, \dots, R^n) \text{ adm.} } 
	|7R^1 
\cap 7R^2_{I(R_1)}  \cap \dots \cap 7R^n_{I(R_1)} |\ \ .
\eeqa
Using (\ref{fiveinfivect}),
which implies that the sets $7R^j_{I(R_1)}$ are nested,
and the estimate (\ref{fivelessfivect}) for the last pair
of sets, we can estimate the last display by 
\beq
\lesssim  \sum_{(R^1, R^2, \dots, R^n) \text{ adm.} } 
	|7R^1 
\cap 7R^2_{I(R_1)}  \cap \dots \cap 7R^{n-1}_{I(R_1)} |\ \ .
\eeq
Iterating the argument allows us to conclude 
(\ref{good}) for $q'$ an integer, which is clearly not a restriction
as the estimate is harder for larger $q'$. 
This completes the proof of Lemma \ref{shadow}.
\end{proof}

\subsection{Construction of the set $G$}

Let $G',H', u$ be as in Lemma \ref{majorsubsetlemma} and define 
$$G_i=\bigcup_{k\in \bbz, k<0}\{
R\in \calr_i: \frac{|E(R)|}{|R|}\ge 2^k  {\  \rm and  \ }
\frac{|H'\cap R|}{|R|}\ge C_\epsilon 2^{-(\frac 12 + \epsilon) k} 
\left(\frac{|H'|}{|G'|}\right)^{\frac 1 2}\}$$
for some small $\epsilon>0$ to be determined later through application of Estimate \ref{secondmaximal}
and some constant $C_\epsilon $ large enough so that we obtain with Theorem \ref{laceylict} below:
\beqa
|G_i| & \leq & \sum _{k\in \bbz, k<0 } C 2^{-k}
\left( C_{\epsilon} 2 ^{-(\frac 12+ \epsilon) k} \left( {{|H|'} \over { |G'|}}  \right) ^{1\over 2}   \right) ^{-2}
|H'| 
\le \frac{|G'|}4\ \ .
\eeqa
This construction essentially allows us to ignore trees with size and density both too large.  This will be used in the proof of Estimate \ref{secondmaximal}.

The following theorem is a variant
of the result in \cite{LL3}.  The theorem there is valid for arbitrary Lipschitz vector fields.  As stated here, the theorem is valid for vector fields depending on one variable.  In fact, the theorem holds for vector fields that are Lipschitz in the vertical direction only.  We recreate the proof given in \cite{LL3} below in the one-variable case.  The only use of the one-variable property comes in the proof of Lemma \ref{uprops} below.
\begin{thm}\label{laceylict}
Let $0\le \delta,\sigma \le 1$, let $H$ be a measurable set,
and let $\calr$  be a finite collection of parallelograms with vertical
edges and dyadic shadow such that for each $R\in \calr$ we have
$$|E(R)|\ge \delta |R|\ ,\ \ |H\cap R|\ge \sigma |R|\ \ .$$
Then 
$$\left|\bigcup_{R\in \calr}R\right|\lesssim \delta^{-1}
\sigma^{-2}|H| \ \ .$$
\end{thm}

\begin{rem}\label{othermaximal}
It is of interest whether a result like Theorem \ref{laceylict} holds 
with $\sigma$- power less than $2$. 
In the single height case, optimal results are already known 
with power all the way to ${1+\epsilon}$; see \cite{B1},\cite{B3}.  
However the important point is that the parallelograms  in Theorem \ref{laceylict} can have arbitrary height, which is 
necessary for creating the exceptional sets needed in the current paper. 
\end{rem}

\begin{proof} 
It is enough to find a subset $\calg\subset \calr$ 
such that
\beqan \label{coverct}
|\bigcup_{R\in \calr }R| & \lesssim &\sum_{R\in \calg} |R|\ ,
\\
\label{twoonect}
\int (\sum_{R\in \calg}1_R)^2 & \lesssim & \delta^{-1} \sum_{R\in \calg} |R|\ .
\eeqan
Namely, we have with (\ref{twoonect})
\beqa \sum_{R\in \calg} |R|
&\le& \sigma^{-1} \int \sum_{R\in \calg} \one_R(x) \one_H\, dx\\
&\le& \sigma^{-1} \|H\|^{\frac 12} (\int (\sum_{R\in \calg} \one_R(x))^2 \, dx)^{\frac 1 2}\\
&\lesssim & \sigma^{-1} \delta^{-\frac 12} |H|^{\frac 12} (\sum_{R\in \calg} |R|)^{\frac 1 2}
\eeqa
and the desired estimate follows from (\ref{coverct}).

We define the set $\calg$ by a recursive procedure. Initialize
\beqa \calg & \leftarrow & \emptyset\ ,\\
STOCK & \leftarrow & \calr\ .
\eeqa
While $STOCK$ is not empty, select $R\in STOCK$
such that $|I(R)|$ is maximal.
Update 
\beqa
\calg & \leftarrow & \calg\cup \{R\}\ ,\\
\calb \leftarrow \{R'\in STOCK: R' &\subset & 
\{x:M_V(\sum_{R\in \calg}1_R)(x)\ge 10^{-3}\}\} \ \ ,\\
STOCK & \leftarrow & STOCK\setminus \calb
\eeqa

This loop will terminate, because
the collection $\calr$ is finite and we remove at each step at least 
the selected $R$ from $STOCK$.

By the Hardy Littlewood maximal bound, it is
clear that (\ref{coverct}) holds and it remains to show
(\ref{twoonect}).
By expanding the square in (\ref{twoonect}) and using symmetry 
it suffices to show
$$\sum_{(R,R')\in\calp}  |R\cap R'| 
\lesssim \delta^{-1}\sum_{R\in \calg}|R|\ \ ,$$
where $\calp$ is the set of all pairs $(R,R')\in \calg\times \calg$
with $R\cap R'\neq\emptyset$ and $R$ is chosen prior to $R'$.
We partition $\calp$ into
\beqa
\calp' &=& \{(R,R')\in \calp: U(R)\not \subset 10^2U(R')\}\ \ ,\\
\calp'' &=&  \{(R,R')\in \calp: U(R) \subset 10^2U(R')\}\ \ .
\eeqa
Theorem \ref{laceylict} is reduced to the following two lemmas: \end{proof}

\begin{lemma}\label{rprimect}
For fixed $R'\in \calg$ we have 
$$\sum_{R\in\calr,  (R,R')\in \calp''}
|R\cap R'| \lesssim |R'|\ .$$
\end{lemma}

\begin{lemma}\label{rct}
For fixed $R\in \calg$  we have
$$\sum_{R'\in\calr,  (R,R')\in \calp'}
|R\cap R'| \lesssim  \delta^{-1}|R|\ .$$
\end{lemma}

Proof of Lemma \ref{rprimect}: 
We first argue by contradiction that $\calp''$
does not contain a pair $(R,R')$ with
$H(R')<H(R)$.
By definition of $\calp''$ we have $U(R)\cap U(100R')\neq \emptyset$.
By Lemma \ref{7rlemma}, applied to $100R_{I(R')}$ and $100R'$, we conclude
that $R'$ is contained in $700R$. 
But then
$$R'\subset \{M_V 1_{R}>1/700\}\ \ ,$$
which contradicts the selection of $R'$
and completes the proof that we have $H(R)\le H(R')$ for all
$(R,R')\in \calp''$.

Now we use Lemma \ref{7rlemma} again
to conclude that for each $(R,R')\in \calp''$ 
we have $R_{I(R')}\subset 700 R'$.
Hence we have for some point $(x,y)$ in $R'$
\beqa 10^{-3} &\ge & M_V(\sum _{R\in \calg : (R,R')\in \calp''}1_{R})(x,y)\\
&\ge & \frac 1{700H(R')}\sum_{R: (R,R')\in \calp''} H(R)\\
&\ge & \frac 1{700} \sum_{R: (R,R')\in \calp''} |R\cap R'|/|R'|\ \ .
\eeqa
This proves Lemma \ref{rprimect}. \endproof

It remains to prove Lemma \ref{rct}.
Fix $R\in \calg$. We decompose $\{R':(R,R')\in \calp'\}$ by the following
iterative procedure:
Initialize
\beqa
STOCK & \leftarrow & \{R': (R,R')\in \calp'\}\ ,\\
\calg' & \leftarrow & \emptyset\ .
\eeqa
While $STOCK$ is non-empty, select $R'\in STOCK$ with maximal
$|I_{R'}|$. Update 
$$\calg' \leftarrow \calg'\cup \{R'\}\ ,$$
$$\calb(R')\leftarrow \{R''\in STOCK \colon \Pi E(R'')\cap \Pi E(R')\neq \emptyset\}\ ,$$
$$STOCK\leftarrow STOCK \setminus \calb(R')\ \ ,$$
where $\Pi$ denotes the projection onto the $x$ axis.
By construction, the sets $\Pi E(R')$ with $R'\in  \calg'$ are disjoint and we have
$$\sum _{R'\in \calg'} |I_{R'}|\le \delta^{-1}\sum _{R'\in \calg'} |\Pi E(R')|\le \delta^{-1} |I(R)|\ .$$
As the sets $\calb(R')$ with $R'\in \calg'$ partition the summation
set of the left-hand-side of Lemma \ref{rct}, it suffices to show for each $R'\in \calg'$
$$\sum _{R''\in \calb(R')} |R''\cap R| \lesssim |R_{I(R')}| \ \ .$$
In what follows we fix $R'\in \calg'$.

\begin{lemma}\label{uprops}
There is an interval $U$ (depending on $R$ and $R'$) of slopes with 
\beqan 
\label{firstuct}
5|U(R)| &\le & |U|  \ \ ,\\ 
\label{seconduct}
 U(R)\cap 5U & = & \emptyset\ \ ,\\
\label{thirduct}
U(R) &\subset &  6U \ \ ,\\
\label{fourthuct}
U(R'') &\subset & U
\eeqan
for all $R''\subset \calb(R')$.
\end{lemma}

\begin{proof}
We distinguish two cases: 
\begin{enumerate}
\item $|U(R)|\le |U(R')|$ 
\item  $|U(R)|>|U(R')|$.
\end{enumerate}
In the first case we use the definition of $\calp'$  
to conclude $$U(R)\cap 25U(R')=\emptyset\ .$$ We then define
$U=KU(R')$ where $K\ge 5$ is the largest number (or very
close to that) such that
$U(R) \cap 5KU(R') =\emptyset$. 
Then we have immediately (\ref{firstuct}), (\ref{seconduct}) 
and (\ref{thirduct}). To see (\ref{fourthuct}) 
assume to get a contradiction that  $U(R'')\not\subset U$. 

By construction of $\calb(R')$, we know that 
$\Pi (E (R'')) \cap \Pi (E(R'))\neq \emptyset$, which implies that 
$U(R'')\cap U(R')\neq \emptyset$ since the underlying vector field $v$ is constant along vertical lines.  Since $U(R')$ is contained in the middle 
fifth of the interval $U$, we conclude $|U|\le 3|U(R'')|$ and $U\subset 7U(R'')$. But then 
$U(R)\subset 10^2 U(R'')$, a contradiction to $(R,R'')\in \calp'$.

In the second case we have $H(R)>H(R')$ because $|I(R')|\le | I(R)|$. 
Since $R'$ is not contained in the set $\{M_V 1_R>10^{-3}\}$ and thus
not in $10^3R$,  we conclude that $U(R')$ contains an element 
not in $400U(R)$. Hence 
$$25 \frac{|U(R)|}{|U(R')|} U(R')$$ does not intersect $U(R)$. From there we may 
proceed as before with $U(R')$ replaced by this bigger interval. 
This completes the proof of Lemma \ref{uprops}. \end{proof}

\begin{lemma}\label{strombergct}
Let $I$ be a dyadic interval contained in $I_{R'}$.
Then for all $R''\in \calb(R')$ with $H(R'')\le 20 |U||I|$ we have
that 
\begin{equation}\label{firstuict}
R_I\cap {R''} \neq \emptyset \implies 
{R''}_I  \subset  
50  (1+{|U||I|} H(R)^{-1}) R 
\end{equation}
and 
\begin{equation}\label{seconduict}
|R_I\cap {R''}|\le  10 |U|^{-1}H(R'') H(R)\ \ .
\end{equation}
\end{lemma}

Proof: By a shearing transformation and translation we may assume that 
the central line segment of $R$ is on the $x$ axis.

Statement (\ref{firstuict}) follows immediately from the central slope
of $R''$ being less than $10|U|$
and  $ H(R'')\le 20 |U||I|$, and hence the vertical distance of 
any point in $R''$ from $R$ is at most $50|U||I|$.
To see the second statement, note that 
the central slope $u_0$ of $R''$ is at least $2|U|$.
Hence (\ref{seconduict}) follows because $R\cap R''$ is contained in a
parallelogram of height $H(R)$ and base $H(R'')u_0^{-1}$.
This proves Lemma \ref{strombergct}. \endproof

\begin{lemma}\label{nosmallct}
Let $I$ be a dyadic interval contained in $I_{R'}$. If
$$\sum_{R''\in \calb(R'): I\subset I_{R''}} 
|R_I\cap R''|>10^{-1} |R_I|$$
then
there does not exist $R'''\in \calb(R')$ with
$I_{R'''}\subset I$, $I_{R'''}\neq I$.
\end{lemma}

\begin{proof} 
For every $R'''\in \calb(R')$ we have $U(R''')\subset U$ and thus
$$ H(R''') \le 10 U |I_{R'''}| \ .$$
Hence if $I_{R'''}\subset I$ then $H(R''')\le 20 U|I|$.
The parallelogram $R'''$ has been selected for $\calg$
after the parallelogram $R$ and the parallelograms $R''\in \calb(R')$ with
$I\subset I_{R''}$.
By Lemma \ref{strombergct} it suffices to show that the maximal 
function 
$$M_V(1_R+
\sum_{R''\in \calb(R'): I\subset I_{R''}} 
1_{R''})$$
is larger than $10^{-3}$ on the parallelogram
$$\tilde{R}:=50(1+{|U||I|}H(R)^{-1}) R\ \ .$$
First assume there exists $R''\in \calb(R')$ with
$I\subset I_{R''}$ and 
$R_I\cap R''\neq \emptyset$
and $H(R'')\ge 20|U||I|$. Note that $U(R'')$ 
and $U(\tilde{R})$ have non-empty intersection because
$U(R'')\subset U\subset U(\tilde{R})$.
Applying Lemma \ref{7rlemma} to the rectangles 
$R''_I$ and $\tilde{R}_{I}$ we obtain similarly
as before 
$$M_V (1_{R''}+1_R)\ge 7^{-1}H(\tilde{R})^{-1} (\min(H(R''),H(\tilde{R}))+H({R})> 10^{-3}$$
on $\tilde{R}_I$, which proves Lemma \ref{nosmallct} in the given case.

Hence we may assume 
$$H(R'')\le 20 |U||I|$$ for every $R''\in \calb(R')$
with $I\subset I_{R''}$ and $R_I\cap R''\neq \emptyset$.
We then have on $\tilde{R}_I$ by Lemma \ref{strombergct}
\beqa M_V (1_R   + \sum_{R''\in \calb(R'): I\subset I_{R''} }\one_{R''})&\ge &
H(\tilde{R})^{-1} (H(R)+\sum_{R''\in \calb(R'): I\subset I_{R''} } H(R''))\eeqa
\beqa &\ge & H(\tilde{R})^{-1} (H(R)+\sum_{R''\in \calb(R'): I\subset I_{R''} } 
|R_I\cap R''| |U| H(R)^{-1})\\
&\ge&  H(\tilde{R}) (H(R)+|U| H(R)^{-1} 10^{-1} |R_I|)\ge 500^{-1} \ \ .
\eeqa
This completes the proof of Lemma \ref{nosmallct}. \end{proof}

Note that we have used the hypothesis $I_{R'''}\neq I$ of Lemma (\ref{nosmallct})
only to conclude that $R'''$ has been selected last to $\calg$. Consider the
collection of all $R''\in \calb(R')$ with $I= I_{R''}$ and let $R'''$ the parallelogram
chosen last in this collection. Since 
$|R_I\cap R'''|\le |R_I|$, the proof  of  the previous
lemma also gives 
\begin{lemma}\label{iequalirct}
We have for every $I\subset I_{R'}$ 
$$\sum_{R''\in \calb(R'):I= I_{R''}} |R_I\cap R''|\le 2 |R_I|\ \ .$$
\end{lemma}

Now let $\cali$ be the set of maximal dyadic intervals contained in $I_{R'}$
such that 
$$\sum_{R''\in \calb(R'):I\subset I_{R''}} |R_I\cap R''|> 2 |R_I|\ .$$
By Lemma \ref{iequalirct} we have $I_{R'}\notin \cali$.
Let $I\in \cali$ and denote the parent of $I$ by $\tilde{I}$.
By Lemma \ref{nosmallct} and by maximality of $I$ and Lemma \ref{iequalirct} we have
\beqa \sum_{R''\in \calb(R')} |R_I\cap R''| &=&
\sum_{R''\in \calb(R'):\tilde{I}\subset I_{R''}} |R_I\cap R''|
+\sum_{R''\in \calb(R'):I=I_{R''}} |R_I\cap R''|\\
&\le &  2 |R_{\tilde{I}}|+  2 |R_I|\le 6 |R_I|\ \ .
\eeqa
By adding over all $I\in \cali$ we obtain
\beq \label{calict}
\sum_{I\in \cali'} 
\sum_{R''\in \calb(R')} |R_I\cap R''| 
\le 6 |R_{I(R')}|\ \ .\eeq

Now let $\cali'$ be the set of maximal dyadic intervals which are contained in
$I_{R'}$, disjoint from any interval in $ \cali$,  and do not contain any
$I(R'')$ with $R''\in \calr(R')$. By construction of $\cali$ 
we have for each  $I\in \cali'$ 
\beqa \sum_{R''\in \calr(R')} |R_I\cap R''|&=&
\sum_{R''\in \calr(R'): I\subset I_{R''}} |R_I\cap R''|\le  2 |R_I|\ \ .
\eeqa
Summing over all intervals in $\cali'$ gives
\beq
\sum_{I\in \cali'} 
\sum_{R''\in \calr(R')} |R_I\cap R''|
\le 2 |R_{(I(R')}|\ \ .
\eeq
Together with (\ref{calict})
this completes the proof of Lemma \ref{rct}, because $\cali$ and $\cali'$
form a partition of $I(R')$.


\section{Outline of the proof of Theorem \ref{firstrefinement}} \label{outline}

Recall that we need to prove for each $|k|\le k_0$ the inequality
\beqan\label{recfirstrefin}
|\langle H_{k,G,H} \one _F , \one _E \rangle| \lesssim
\left( {{ |G| } \over {|H|} }\right) ^{{1\over 2}-{1 \over p}}  
|F|^{{1\over 2} } |E|^{{1\over 2}}\ \ .\eeqan
We assume without loss of generality that $E\subset G$ and $F\subset H$.
Recall also that Theorem \ref{companion} implies for  $1<q<\infty$:
\beqan \label{b2result}
|\langle H_{k} \one _{F} , \one _{E} \rangle| 
&\lesssim &   
\left( {{ |E| } \over {|F|} }\right) ^{{1\over 2}-{1 \over q}}  
|F |^{{1\over 2} } |E |^{{1\over 2}}
\ \ .
\eeqan
The left hand sides of (\ref{recfirstrefin}) and (\ref{b2result}) are identical.
Hence our task is to strengthen the proof of Theorem \ref{companion} in \cite{B2} 
in case the factor involving $G$ and $H$ in (\ref{recfirstrefin}) is less
than the corresponding factor involving $E$ and $F$ in (\ref{b2result}).

We recall some details about the proof in \cite{B2}.
The form $\langle H_k \one_{F}, \one_{E}\rangle$ is written 
as a linear combination of a bounded number of model forms
$$\sum_{s\in \calu_k}\langle {\bf C}_{s,k} \one_{F}, \one_{E}\rangle\ \ ,$$
where the index set $\calu_k$ is a set of parallelograms with vertical edges
and constant height (depending on $k$). The paper proves the bound analogous to (\ref{b2result})
 for the absolute sum 
\beq \label{modelsumb}
\sum_{s\in \calu'_k}|\langle {\bf C}_{s,k} \one_{F}, \one_{E}\rangle|\ \ ,
\eeq
where $\calu'_k$ is an arbitrary finite subset of $\calu_k$ and the bound is independent
of the choice of subset, which may be assumed to only account for non-zero summands.

To estimate (\ref{modelsumb}), one first proves estimates for the sum over
certain subsets of $\calu'_k$ called trees. Each tree $T$ is assigned a parallelogram ${\topp}(T)$.
It is also assigned a density $\delta(T)$ which measures the contribution of $E$ to the tree,
and a size $\sigma(T)$ which measures the contribution of $F$ to the tree.
One obtains for each tree $T$:
$$\sum_{s\in T}|\langle {\bf C}_s \one_{F} , \one_{E}\rangle|
\lesssim \delta (T)\sigma (T) |{\topp}(T)|\ \ .$$

The collection $\calu'_k$ is then written as a disjoint union of sub-collections  $\calu_{\delta,\sigma}$ 
where $\delta$ and $\sigma$ run through the set of integer powers of two.
Each $\calu_{\delta,\sigma}$ is written as a disjoint union of a collection $\calt_{\delta,\sigma}$ of
trees with density at most $\delta$ and size at most $\sigma$.
With the above tree estimate it remains  to estimate
$\sum_{\delta,\sigma} S_{\delta,\sigma}$ with
$$S_{\delta,\sigma}:=\sum_{T\in \calt_{\delta,\sigma}} \delta \sigma |{\topp}(T)|\ \ .$$
We list the estimates on $S_{\delta,\sigma}$ used in \cite{B2}; note that we include an additional factor 
of $\delta\sigma$ relative to the corresponding expressions in \cite{B2}.

\begin{estimate}[Orthogonality]\label{orthogonalityct}
$S_{\delta,\sigma}\lesssim {|F| \delta \sigma^{-1}}\ \ .$
\end{estimate}

\begin{estimate}[Density]\label{densityct}
$S_{\delta,\sigma}\lesssim {|E|\sigma}\ \ .$
\end{estimate}

\begin{estimate}[Maximal]\label{maximalct}
For any $\epsilon>0$,
$S_{\delta,\sigma}\lesssim {|F|^{1-\epsilon}|E|^{\epsilon}
\sigma^{-\epsilon}}\ \ .$
\end{estimate}

\begin{estimate}[Trivial density restriction]\label{trivialdensity}
If $\delta>1$, then
$S_{\delta,\sigma}=0\ \ .$
\end{estimate}

\begin{estimate}[Trivial size restriction]\label{trivialsize}
There is a universal $\sigma_0$ such that if $\sigma>\sigma_0$, then
$S_{\delta,\sigma}=0\ \ .$
\end{estimate}

Our improvement comes through two additional estimates depending 
on $G$ and $H$ that will be proved in Section \ref{fillingdetails}.

\begin{estimate}[Second maximal]\label{secondmaximal}
If $p<2$ and $G$, $H$ are as in Theorem \ref{firstrefinement}, then
for every $\epsilon>0$
$$S_{\delta,\sigma}
\lesssim |E| \left({|H|\over |G|}\right)^{{1\over 2}} 
\sigma^{-\epsilon} \delta^{-{\frac 12}-
\epsilon}\ \ .$$
\end{estimate}

\begin{estimate}[Size restriction]\label{sizerestriction}
Let $p>2$ and $G$, $H$ as in Theorem \ref{firstrefinement}.
Let $n>2$ be a large integer and $\alpha=1/n$ and $C_\alpha$ be some constant.
Then there is a constant $\sigma_1$ such that if 
$$\sigma \ge \sigma_1 \left({\tilde{\delta} \over \delta}\right)^n$$
with 
$$\tilde{\delta}=C_{\alpha} \left({|G|\over |H|}\right)^{1-\alpha}\ \ ,$$
then we have
$S_{\delta,\sigma}=0\ \ .$
\end{estimate}

To obtain summability for small $\sigma$, it is convenient to take
weighted geometric averages of Estimates \ref{orthogonalityct}, 
\ref{maximalct}, and \ref{secondmaximal} with Estimate \ref{densityct} to
obtain positive powers of $\sigma$. We record these modified estimates,
where we simplify exponents using that we may assume universal upper bounds on
$\delta$ and $\sigma$. We have for any $\epsilon>0$:

\begin{estimate}[Modified Orthogonality]\label{modorthogonalityct} $S_{\delta,\sigma}\lesssim {|E|^{\frac 12 +\epsilon }
|F|^{\frac 12 -\epsilon} \delta^{\frac 12   -\epsilon} \sigma^{2\epsilon}}\ \ .$
\end{estimate}

\begin{estimate}[Modified maximal]\label{modmaximalct}
$S_{\delta,\sigma}\lesssim {|F|^{1-4\epsilon}|E|^{4\epsilon}
\sigma^{\epsilon}}\ \ .$
\end{estimate}

\begin{estimate}[Modified Second maximal]\label{modsecondmaximal}
Under the assumptions of Estimate \ref{secondmaximal},
$$S_{\delta,\sigma}
\lesssim |E| \left({|H|\over |G|}\right)^{{1\over 2}-\epsilon} 
\sigma^{\epsilon} \delta^{- {1 \over 2}}\ \ .$$
\end{estimate}

In the rest of this section we show how these estimates 
are used to estimate $\sum_{\delta,\sigma} S_{\delta,\sigma}$ and thereby complete 
the proof of Theorem \ref{firstrefinement}.

\subsection{Case $p<2$ and $|H|\le |G|$.}
Inequality (\ref{recfirstrefin}) for  ${3\over 2}<p<2$
follows from inequality (\ref{b2result}) for $1<q<2$ unless
\beqan\label{efgh1}
\left({|H |\over|G|}\right)^{1\over 3}\le
{|F|\over|E|}\ \ ,
\eeqan
which we shall therefore assume.

Pick $\epsilon>0$ small compared to the distance of $p$ to $\frac 32$. 
We split the sum over $\delta$ at 
$$\delta_0=\left( {|H| \over |G|}{|E| \over |F|}\right)^{1\over 2}\ .$$
For $\delta\le \delta_0$ we use Estimate \ref{modorthogonalityct} together
with Estimate \ref{trivialsize} to obtain
\beqa \sum_{\delta\le \delta_0}\sum_{\sigma}  S_{\delta,\sigma}
&\lesssim & \delta_0^{{1\over 2}-\epsilon} 
|E|^{{1\over 2}+\epsilon}|F|^{{1\over 2}-{\epsilon}}
= 
|E|^{{3\over 4}+{\epsilon\over 2}}|F|^{{1\over 4}-{\epsilon\over 2}}
\left({|H|\over |G|}\right)^{{1 \over 4}-{\epsilon \over 2}}\ \ .
\eeqa
For $\delta\ge \delta_0$ we use Estimate \ref{modsecondmaximal} together
with Estimate \ref{trivialsize} to obtain
\beqa \sum_{\delta\ge \delta_0}\sum_{\sigma }  S_{\delta,\sigma}
&\lesssim &  
\delta_0^{-{1\over 2}} |E|\left({|H|\over |G|}\right)^{\frac 12 -\epsilon}
=  
 |E|^{{3\over 4} }|F|^{{1\over 4}} 
\left({|H|\over |G|}\right)^{{1\over 4}-\epsilon}\ \ .\\
\eeqa
Using (\ref{efgh1}) and $|H|\le |G|$ we may estimate both partial sums by
$$\lesssim |E|^{1\over 2}|F|^{1\over 2}\left({|H|\over |G|}\right)^{{1\over 6}-3\epsilon}\ .$$
This completes the proof of (\ref{recfirstrefin}) in case $p<2$.

\subsection{Case $p>2$ and $|G|\le |H|$.}

Pick $\epsilon$ very small compared to $\frac 1p$.
Inequality (\ref{recfirstrefin}) for  $2<p<\infty$
follows from inequality (\ref{b2result}) unless
\beqan\label{efgh2}
{|G|\over |H|}
\le 
\left({|E|\over |F|}\right)^{1+\epsilon}
\ \ ,
\eeqan
which we shall therefore assume.
Let $\alpha$ and $1/n$ be very small compared to $\epsilon$, let $C_\alpha$
be as in the construction of the set $H$ and let $\tilde{\delta}$ be as in 
Estimate \ref{sizerestriction}.
We split the sum over $\delta$ at 
$$\delta_1:=\tilde{\delta} \left({1\over \tilde{\delta}} {|E|\over |F|}\right)^{1\over n}\ \ .$$

For $\delta\le \delta_1$ we use 
a weighted geometric mean of Estimates
 \ref{modorthogonalityct} and \ref{modmaximalct} 
together with Estimate \ref{trivialsize} to obtain
\beqa \sum_{\delta\le \delta_1}\sum_{\sigma} S_{\delta,\sigma}
&\lesssim & \delta_1^{{1\over 2}-4\epsilon} 
|E|^{{1\over 2}-\epsilon}|F|^{{1\over 2}+{\epsilon}}
\\
&\lesssim & 
\tilde{\delta}^{(1-{1\over n})({1\over 2}-4\epsilon)} |E|^{1\over 2}|F|^{{1\over 2}}
\left({|G|\over |H|}\right)^{-2\epsilon}\ \ ,
\eeqa
where in the last line we have used (\ref{efgh2}) and $|G|\le |H|$. Using the definition
of $\tilde{\delta}$ in Estimate \ref{sizerestriction} we may estimate the last display by
\beq
\lesssim  
|E|^{1\over 2}|F|^{{1\over 2}}
\left({|G|\over |H|}\right)^{{1\over 2}-10\epsilon}\ \ .
\eeq
For $\delta\ge \delta_1$ we use 
Estimate \ref{densityct} together with
Estimate \ref{sizerestriction} to obtain
\beqa \sum_{\delta\ge \delta_1}\sum_{\sigma }  S_{\delta,\sigma}
&\lesssim & \sum_{\delta\ge \delta_1}  (\tilde{\delta}/\delta)^n |E|
\lesssim  (\tilde{\delta}/\delta_1)^n |E|\\
&\lesssim&   \tilde{\delta}  |F| \lesssim   
|F|^{1\over 2}|E|^{1\over 2} \left({|G|\over |H|}\right)^{{1\over 2}-10\epsilon}\ \ .
\eeqa
where in the last line we have used (\ref{efgh2}) and 
$|G|\le |H|$.
This completes the proof of (\ref{recfirstrefin}) in case $p>2$.


\section{Proof of the additional Estimates \ref{secondmaximal} and \ref{sizerestriction}}\label{fillingdetails}

In this section we deviate from the notation in Section \ref{constructsets} as follows: for a
parallelogram $R$ we denote by $CR$ the isotropically scaled
parallelogram with the same center 
and slope as $R$ but with height $H(CR)=CH(R)$ and shadow $I(CR)=CI(R)$.

We say that a set is approximated by a parallelogram $R$, if it is contained in 
the parallelogram and the parallelogram has at most one hundred times the area
of the set. Any parallelogram $R$ can be approximated by a parallelogram
$R'$ with $I(R')\in \cali_1\cup \cali_2$ and both vertical edges
of $R'$ in $\cali_1\cup\cali_2$. To see this, first identify an interval 
$I$ in $\cali_1\cup \cali_2$ which contains $I(R)$ and has at most three times the length;
this interval $I$ will be the shadow of $R'$. Consider the extension of $R$ which has same 
central line and height as $R$ but shadow $I$. Then find two intervals in $\cali_1\cup \cali_2$ 
which have mutually equal length at most three times the height of $R$ and which contain
the respective vertical edges of the extended parallelogram. These intervals define
the vertical edges of $R'$.

We recall some details of the proof of Estimate \ref{densityct}
in \cite{B2}.
Given $\delta,\sigma$, one constructs a collection
$\calr_{\delta,\sigma}$ of parallelograms of the same height
as the parallelograms in $\calu'_k$ such that each
tree $T$ in $T_{\delta,\sigma}$ is assigned a parallelogram $R$ 
in $\calr_{\delta,\sigma}$ with $\topp(T)\subset C_0R $ and
$\topp(T')\subset C_0R $ for
every sub-tree $T'$ of $T$, for some constant $C_0$.
If $\calt(R)$ denotes the trees
in $\calt_{\delta,\sigma}$ which are assigned a given parallelogram $R\in\calr_{\delta,\sigma}$, then
we have 
$$\sum_{T\in \calt(R)}|\topp(T)|\le C_1|R|$$
for some constant $C_1$.
Estimate \ref{densityct} is then deduced from the inequality
\beq\label{rdisjoint}
\sum_{R\in \calr_{\sigma,\delta}}|R |\lesssim {|E|}{\delta^{-1}}\ \ .
\eeq
which follows essentially from pairwise incomparability of the parallelograms
in $\calr_{\delta,\sigma}$.  (In other words, if two parallelograms $P_1$, $P_2$ overlap, then they are pointed in different directions, resulting in disjointness of the sets $E(P_1)$ and $E(P_2)$.)  All parallelograms in $\calr_{\delta,\sigma}$
have height at least $2^{-k_0}$, length of shadow at least $2^{-k_0}$,
and slope at most $10^{-1}$.

Let $Q=[-N,N]^2$ be the large square with $N$ as in Lemma \ref{majorsubsetlemma}.
We claim that every set $Q\cap 2^k R$ with $R\in \calr_{\delta,\sigma}$ and $k\ge 0$ 
can be approximated by a parallelogram in $\calr_1\cup \calr_2$.
If $Q\cap 2^kR$ is a parallelogram then this is clear by the remarks above.
If $Q\cap 2^kR$ is not a parallelogram, then we first extend it to the 
minimal parallelogram containing it, which thanks to the 
bounded slope of $R$ is not much larger than $Q\cap 2^kR$, and then approximate
the extension by a parallelogram in $\calr_1\cup \calr_2$.

\subsection{Proof or Estimate \ref{secondmaximal}}

We partition $\calr_{\delta,\sigma}$ into subset $\calr_{\delta,\sigma,j}$ 
consisting of all parallelograms in $\calr_{\delta,\sigma}$ such that
$$C_12^{-j-1}|R|\le \sum_{T\in \calt(R)} |{\bf top}(T)|< C_1 2^{-j}|R|\ \ .$$
We claim that $\calr_{\delta,\sigma, j}$ is empty unless
$j$ satisfies $(\ref{sigmaj})$ below.
This claim together with (\ref{rdisjoint}) will prove Estimate \ref{secondmaximal}:
\beqa
S_{\delta,\sigma} & \lesssim & 
\delta \sigma 	\sum_{j_0\lesssim j } 
\sum_{\calr_{\delta, \sigma, j} }  2^{-j} |R| 
 \lesssim  \sum_{j_0\lesssim j }    2^{-j}{{|E|} {\sigma } }  \\
& \lesssim & |E|\sigma^{-\epsilon} \delta ^{-\frac 12 \epsilon } \left(\frac {|H|} {|G| }  \right)^{1\over 2} \ \ .
\eeqa

It remains to prove the claim.
Suppose there is a parallelogram $R$ in $R\in \calr_{\sigma,\delta,j}$. It has large 
density as defined and discussed in \cite{B2}, which implies that there is a $k\ge 0$ with
$$|E(2^k R)\cap G| \ge 2^{20k}\delta|2^kR|\ \ .$$
Since $G$ is contained in $Q$, we may approximate $Q\cap 2^kR$ by a 
parallelogram $R'$ of $\calr_1\cup \calr_2$ and obtain
\beqan\label{erprime}
|E(R')|\ge |E(R')\cap G| \gtrsim  2^{20k}\delta|R'|\ \ .
\eeqan

Now suppose first that $2^k\ge \sigma^{-\epsilon}$. 
By Claim 18 in \cite{B2}, and using that
$F\subset Q$, we obtain
$${|F\cap H \cap R'|\over |R'|}
\gtrsim 
{|F\cap H \cap 2^kR|\over |2^kR|}
\gtrsim 
2^{-2k} 2^{-j}\sigma^{1+\epsilon}\ \ .$$
On the other hand, (\ref{erprime}) implies in particular $R'\cap G\neq \emptyset$, which by construction of
$G$ (see Section \ref{constructsets}) implies, using $k\ge0$:
$$
2^{-2k} 2^{-j} \sigma ^{1+ \epsilon}\lesssim 
 (2^{20k}\delta) ^{-(\frac 12+\epsilon)} \left(\frac  {|H|} {|G|} \right) ^{1\over 2}  
\ \ ,$$
\beq \label{sigmaj} 2^{-j} \lesssim 2^{-j_0}:=\sigma ^{-1-\epsilon}
 \delta ^{-\frac 12 - \epsilon } \left( {{|H|} \over {|G|}}\right) ^{1\over 2}  \ \ .
\eeq
If $2^k\le \sigma^{-\epsilon}$ we use the variant
$${|F\cap H\cap \sigma^{-\epsilon}R|\over |\sigma^{-\epsilon} R|}\ge 2^{-j} \sigma^{1+3\epsilon}$$
of Claim 18 in \cite{B2} to obtain the same conclusion.

\subsection{Proof of Estimate \ref{sizerestriction}}

Note that by Estimates \ref{trivialdensity}
and $\ref{trivialsize}$ we may assume $C_0\tilde{\delta}\le  \delta$
with $C_0$ as above.
Suppose $T_{\delta,\sigma}$ is non-empty.
Consider a tree $T$ in $\calt_{\delta,\sigma}$
and let $R\in \calr_{\delta,\sigma}$ be the associated parallelogram
as above. As above we have for some $k\ge 0$:
$$|E(2^k R)\cap G|\ge 2^{20k}\delta |2^kR|\ \ .$$
Define $m$ so that $\delta$ is within a factor two of  $C_0^2 2^{m}\tilde{\delta}$
and note that $m\ge 0$.
Let $R'\in \calr_1\cup \calr_2$ be an
approximation of $Q\cap \max (2^k, C_0 2^m) R$. We then have
$$|E(R')\cap G|\ge \tilde{\delta} |R'|\ \ .$$
By construction, $R'$ is disjoint from $H$.
Since $\topp(T)$ is contained in $C_0R$, we have that 
$2^m \topp(T)$ is contained in $R'\cup Q^c$, and the same holds
with $T$ replaced by any sub-tree $T'$ of $T$.

But by Lemma 29 of \cite{B2} with $f=\one_{F\cap H}$, 
we obtain with the notation in that Lemma for every sub-tree $T'$ of $T$:
\beqa \sum_{s\in T'}|\langle f, \phi_s\rangle|^2 &=&
\sum_{m'\ge m}\sum_{s\in T'}|\langle f \one_{2^{m'+1}\topp(T')\setminus 2^{m'}\topp(T')}, \phi_s\rangle|^2\\
&\lesssim & \sum_{m'\ge m} 2^{-4n m'} \|f 1_{2^{m'+1}\topp(T')}\|_2^2\\
&\lesssim & 2^{-2n m} | \topp(T')|\ \ .
\eeqa
By the definition of $\sigma(T)$ this implies
$$\sigma(T)\le 2^{-n m}\ \ ,$$
which in turn implies Estimate \ref{sizerestriction}.


\bigskip

\footnotesize
\noindent\textit{Acknowledgments.}

The authors thank Ciprian Demeter for explaining some background for this problem and discussing
various approaches to this problem. The first author was partially supported by NSF grant DMS 0902490. 
The second author was partially supported by NSF grants DMS 0701302 and DMS 1001535.

\end{document}